\documentclass[sn-mathphys-num]{sn-jnl}
\usepackage{amssymb,amsmath,amsthm}
\usepackage{enumerate}
\usepackage{verbatim}
\usepackage{antpolt}
\newtheorem{thm}{Theorem}[section]

\newtheorem{prop}[thm]{Proposition}
\theoremstyle{definition}

\newtheorem{defn}[thm]{Definition}
\theoremstyle{remark}

\DeclareMathOperator{\Gr}{Gr_{res}}
\DeclareMathOperator{\Gro}{Gr_{res}^0}
\DeclareMathOperator{\Grp}{Gr_{res}^p}
\newcommand{\U}{U}
\newcommand{\Ur}{U_{\textrm{res}}}
\newcommand{\UU}{\mathfrak U(\H)}
\newcommand{\UUr}{\mathfrak U_{\textrm{res}}}
\newcommand{\uu}{\mathfrak u}
\newcommand{\ur}{\mathfrak u_{\textrm{res}}}

\newcommand{\emphh}[1]{{\bf #1}}
\newcommand{\gr}{\rightrightarrows}
\newcommand{\Li}{L^\infty(\H)}
\renewcommand{\L}{\mathcal{L}(\H)}

\newcommand{\be}{\begin{equation}}
\newcommand{\ee}{\end{equation}}
\newcommand{\ben}{\begin{enumerate}}
\newcommand{\een}{\end{enumerate}}
\newcommand{\norm}[1]{\left\Vert#1\right\Vert}

\newcommand{\prf}[1]{\begin{proof}#1\end{proof}}
\renewcommand{\to}{\rightarrow}
\newcommand{\tto}{\longrightarrow}
\renewcommand{\H}{\mathcal{H}}

\DeclareMathOperator{\id}{id}

\begin{document}
  \title
  %[Banach Lie groupoid of partial isometries$\ldots$]
  {Banach Lie groupoid of partial isometries over restricted Grassmannian}
  %\author{Tomasz Goli\'nski, Grzegorz Jakimowicz, Aneta Sliżewska}
  %\dedicatory{University of Bia{\l}ystok\\
  %   Faculty of Mathematics\\
  %   Ciołkowskiego 1M, 15-245 Bia{\l}ystok, Poland\\
  %   email: tomaszg@math.uwb.edu.pl, g.jakimowicz@uwb.edu.pl, anetasl@uwb.edu.pl
  % }
  \author*[1]{\fnm{Tomasz} \sur{Goliński}}\email{tomaszg@math.uwb.edu.pl}
\author[1]{\fnm{Grzegorz} \sur{Jakimowicz}}\email{g.jakimowicz@uwb.edu.pl}
\author[1]{\fnm{Aneta} \sur{Sliżewska}}\email{anetasl@uwb.edu.pl}

\affil*[1]{\orgdiv{Faculty of Mathematics}, \orgname{University of Bia{\l}ystok}, \orgaddress{\street{Ciołkowskiego 1M}, \city{Białystok}, \postcode{15-245}, \country{Poland}}}

\abstract{The set of partial isometries in a $W^*$-algebra possesses a structure of Banach Lie groupoid. In this paper the differential structure on the set of partial isometries over the restricted Grassmannian is constructed, which makes it into a Banach Lie groupoid.
}


\maketitle

\tableofcontents
  
\section{Introduction}

The theory of Lie groupoids and algebroids in infinite-dimensional setting is still very young. To our knowledge, the first results in the Banach context were obtained in \cite{Anastasiei} and \cite{pelletier}. Further results can be found in \cite{BGJP}, where the definitions of the notions of Banach Lie groupoid and algebroid %taking account of the subtleties of infinite dimensional case 
have been proposed and many analogues of finite dimensional results were proven. 
There was also some work done in a more general contexts, see e.g. \cite{schmeding15,glockner18,pelletier19,schmeding22}.

An interesting class of examples of Banach Lie groupoids was presented in the paper \cite{OS}. Namely a structure of Banach Lie groupoid was constructed on the set of partial isometries $\mathcal U(\mathfrak M) \gr \mathcal L(\mathfrak M)$ and the set of partially invertible elements $\mathcal G(\mathfrak M) \gr \mathcal L(\mathfrak M)$ of a $W^*$-algebra $\mathfrak M$. The base of these groupoids is the lattice $\mathcal L(\mathfrak M)$ of projectors of a given $W^*$-algebra. In follow-up papers \cite{OJS,OJS-fiber} the Banach Lie algebroid and Poisson side of the picture was investigated. 
One obtains a particular example of this construction by taking as $\mathfrak M$ the $W^*$-algebra of bounded operators acting in a separable complex Hilbert space $L^\infty(\H)$. For simplicity we will denote the groupoid of partial isometries in that case by $\UU \gr \L$. This groupoid turned out to be useful for the study of differential structure of unitary orbits of normal operators, see \cite{beltita23}. Moreover in the papers \cite{GT-momentum,GT-partiso-bial} the hierarchy of differential equations on $\UU$ was discussed.

Generally one can say that there are not many examples of Banach Lie groupoids and algebroids in the literature so far. The purpose of this paper is to investigate a new example: the groupoid of partial isometries $\UUr \gr \Gr$ over the restricted Grassmannian $\Gr$. From an algebraic point of view it is a subgroupoid of $\UU \gr \L$, but the topology and manifold structure need to be different. Another motivation is the hope that this groupoid will have applications in the integrable systems related to the restricted Grassmannian introduced in \cite{GO-grass}.

The restricted Grassmannian $\Gr$ (known also as the Sato Grassmannian) is a strongly symplectic (or even K\"ahler) manifold modelled on a Hilbert space. It first appeared in the study of the KdV and KP equations, see \cite{sato-sato,segal-wilson}. It has however many other applications including quantum field theory \cite{spera-valli,wurzbacher,chiumiento24} or loop groups \cite{segal,sergeev}. It also provides non-trivial examples of Banach Lie--Poisson spaces, which are useful for studying integrable Hamiltonian systems, see \cite{Oext, Ratiu-grass, GO-grass, GO-grass2, GT-momentum}, and it is also connected with Banach Poisson--Lie groups \cite{tumpach-bruhat} and Siegel disc \cite{T-siegel}. As such it is an important object in infinite dimensional differential geometry, see also \cite{tumpach-hyperkahler,andruchow08,zelikin} for a wider context. 

The structure of the paper is as follows. For the sake of self-consistency we start by recalling basic notions and results, which will be needed later on. In Section \ref{sec:BLgr} the definition of Banach Lie groupoid is presented in a particular case of Hausdorff manifolds (see \cite{BGJP} for general case). Next in Section \ref{sec:partiso} the construction of the groupoid of partial isometries in the Hilbert space is described (see \cite{OS}). Section \ref{sec:grres} recalls the definition and basic facts about the restricted Grassmannian $\Gr$. The main results of the paper are contained in Section \ref{sec:res-group}, where the set $\UUr$ of partial isometries over $\Gr$ is introduced and the structure of Banach Lie groupoid is defined on it. Since $\Gr$ is a homogeneous space (unlike $\L$) this construction is carried out using a family of local cross-sections.

\section{Banach Lie groupoids}\label{sec:BLgr}

Before we start let us fix the terminology related to Banach geometry that will be used later on. Note that different conventions are used in literature in general. In our paper a smooth map $f:N\to M$ between two Banach manifolds will be called a \emph{submersion} if for each $x\in N$ the tangent map $Tf: T_xN \to T_{f(x)}M$ is a surjection and $\ker T_xf$ is a split subspace of $T_xN$.

Let us recall the definition of the Banach Lie groupoid from the paper \cite{BGJP}. For simplicity we restrict our attention here to Hausdorff case, which is sufficient in the discussed situation. For a more detailed discussion we refer to aforementioned paper.

\begin{defn}\label{BLG}
A Banach Lie groupoid $\mathcal{G}\gr M$ is a pair $(\mathcal{G},{M})$ of Banach manifolds and the following maps: 
\begin{itemize}
\item surjective submersions
$s: \mathcal{G}\to M$ and $t: \mathcal{G}\to M$ 
called \emph{source} and \emph{target} maps, respectively.

\item smooth map $m:\mathcal{G}^{(2)}\to\mathcal{G}$, where $\mathcal{G}^{(2)}$ is the set of composable elements and is defined as
\[\mathcal{G}^{(2)}:=\{(g,h)\in\mathcal{G}\times\mathcal{G}\;|\; s(g) = t(h)\}\]
and endowed with the induced topology from the product $\mathcal{G}\times \mathcal{G}$. This map is called a \emph{multiplication} and denoted by $m(g,h)=gh$. Associativity condition is required in the sense that the product $(gh)k$ is defined if and only if $ g(hk)$ is defined and in this case they coincide.

\item continuous embedding $\epsilon: M \to \mathcal{G}$ called \emph{identity section} such that
$g\epsilon(x) = g$ for all $g\in s^{-1}(x)$, and $\epsilon(x) g = g$ for all $g\in t^{-1}(x)$.

\item diffeomorphism $\iota: \mathcal{G}\to \mathcal{G}$,  
% denoted $\iota(g)=g^{-1}$ 
called \emph{inversion}, which satisfies
$g\iota(g)=\epsilon(t(g))$, $\iota(g)g=\epsilon(s(g))$ for all $g\in\mathcal G$.
\end{itemize}
The manifold $M$ is called the base of the groupoid, and $\mathcal{G}$ is called the total space of the groupoid.
\end{defn}

It was proved in \cite[Proposition 3.1]{BGJP} (see also \cite{glockner15}) that if those conditions hold, the following facts are true:
\ben
\item $\mathcal G^{(2)}$ is a submanifold of $\mathcal G\times \mathcal G$, 
\item the map $\epsilon$ is smooth,
\item $\epsilon(M)$ is a closed Banach submanifold of $\mathcal G$.
\een

Moreover conditions on $t$ can be relaxed as it is automatically a surjective submersion given that $s$ is a surjective submersion and $\iota$ is a diffeomorphism as it uniquely defined by them via $t=s\circ\iota$. 

Standard examples of Banach Lie groupoids include Banach Lie groups, pair groupoids on a Banach manifolds, general linear Banach Lie groupoids of a vector bundle or the the action groupoids of a smooth action of a Banach Lie group on a Banach manifold. As said before, there are Banach Lie groupoids associated to $W^*$-algebras and even unital associative Banach algebras. We will present a particular groupoid from this class in the next section.

For each Banach Lie groupoid there exists the canonically associated Banach Lie algebroid, but we will not present this construction and the necessary definitions here, see \cite{BGJP}.

\section{Banach Lie groupoid of partial isometries}\label{sec:partiso}

In this section we recall some of the results of the paper \cite{OS} adapting them to the particular case of $W^*$-algebra $L^\infty(\H)$ of bounded operators on a Hilbert space $\H$. We focus only on the groupoid of partial isometries $\mathcal U(\Li) = \UU$. The other groupoid constructed in the paper (groupoid of partially invertible elements) will be subject of a separate study.

Let $\L$ denote the lattice of orthogonal projectors in Hilbert space~$\H$
\[\L = \{ p \in \Li \;|\; p^2 = p^* = p\}.\]
For our purposes it will be useful to identify the projector with its image. In this setting $\L$ can be seen as the Grassmannian of all closed subspaces of $\H$. 

The set $\UU$ of partial isometries acting on $\H$ is defined as
\[ u\in\UU \iff u^*u \in \L \iff uu^* \in \L.\]
Alternatively they can be described as unitaries between two closed subspaces of $\H$ (more precisely the orthogonal complement of their kernel and their image) and they fulfill the following equalities
\[u^* u u^* = u^*\qquad uu^*u=u.\]

From the definition it is straightforward that orthogonal projectors are partial isometries $\L\subset \UU$.

In general the product of two partial isometries is not a partial isometry. However introducing appropriate source and target maps one can discover the groupoid structure $\UU\gr \L$. It was introduced in \cite{OS} and it is given by the following natural maps:
\begin{align}
s(u) &= u^*u,\label{u_source}\\
t(u) &= uu^*,\label{u_target}\\
m(u_1,u_2) &= u_1u_2,\label{u_mult}\\
\epsilon(p) &= p,\label{u_id}\\
\iota(u) &= u^*\label{u_inv},
\end{align}
for $u, u_1, u_2 \in \UU$ such that $u_2u_2^* = u_1^*u_1$ and $p\in \L$.
In this case the condition $u_2u_2^* = u_1^*u_1$ implies that the product $u_1u_2$ is again a partial isometry.

% The groupoid $\mathcal U \gr \mathcal L$ is a topological groupoid with respect to the operator topology of $L^\infty(\H)$. However 

In order to obtain Banach Lie groupoid structure on $\UU$ one needs 
%to work with a different topology. It is obtained by 
first to describe a differential structure on the Grassmannian $\L$. In the paper \cite{OS} it is constructed in terms of $W^*$-algebras. For the sake of this paper a more direct approach is advised, namely a version of the construction for the case of split Grassmannian of a Banach space from \cite[Section 3.1.8.G]{ratiu-mta}. We will present more straightforward and readable formulas that can be obtained in this particular case.

The construction of this differential structure goes through a family of charts $\phi_W: U_W \to L^\infty(W,W^\perp)$ indexed by $W\in\L$ defined by
\be \label{chart_grass}\phi_W(V) = P_{W^\perp} (P_W|_{V})^{-1},\ee
where
\be \label{chart_dom_grass} U_W = \{ V \in \L \;|\; V\oplus_B W^\perp = \H \},\ee
$P_W$ is orthogonal projection on $W$ and $\oplus_B$ denotes a direct sum in the sense of Banach spaces (i.e. not necessarily orthogonal). Note that the condition for $V$ to belong to the chart domain $U_W$ is equivalent to invertibility of the projection $P_W|_{V}$.

The inverse map  $\phi_W^{-1}: L^\infty(W,W^\perp) \to \L$ to a chart assigns to a bounded operator its graph in $W\oplus W^\perp = \H$ i.e.
\be \phi_W^{-1}(A) = \{ (w, Aw) \in W\times W^\perp \;|\; w\in W\}\ee
for $A \in L^\infty(W,W^\perp)$.
The transition functions can be computed explicitly, see e.g. \cite{ratiu-mta,segal,OS} and are homographies.

One can also consider $\L$ as a level set inside $\Li$ and use this point of view to define a differential structure, see e.g. \cite{porta1987,andruchow2005}. It is straightforward to check that it would produce the same differential structure.

The differential structure on $\UU$ compatible with the groupoid structure was described in details in \cite{OS}. We will not present the explicit formulas here. Note that one can adapt the expressions from the Section \ref{sec:res-group}, one only needs to be careful as $\UU$ does not act transitively on $\L$ and in consequence one needs to consider each orbit enumerated by dimension and codimension of the subspace separately.

\section{Restricted Grassmannian as a Hilbert manifold}\label{sec:grres}

In this section we will recall briefly the necessary information about the restricted Grassmannian. More detailed exposition can be found e.g. in \cite{segal,wurzbacher}.

We fix an orthogonal decomposition (called polarization) of the Hilbert space $\H$ 
\be\label{polarization}\H=\H_+\oplus \H_-\ee
onto infinite dimensional closed subspaces $\H_\pm$. We will denote by $P_+$ and $P_-$ the orthogonal projectors onto $\H_+$ and $\H_-$ respectively. %In general one assumes that both Hilbert subspaces $\H_\pm$ are infinite dimensional, but some results might also be obtained otherwise.

$L^p(\H)$ will denote the Schatten class of operators acting in $\H$ equipped with the norm $\norm{\,\cdot\,}_p$. The $L^p(\H)$ spaces are ideals in algebra $\Li$. In particular $L^1(\H)$ denotes the ideal of trace-class operators and $L^2(\H)$ is the ideal of Hilbert--Schmidt operators, which possesses additionally the structure of Hilbert space. By $L^0(\H)\subset \Li$ one denotes the ideal of compact operators, which is operator norm $\norm\cdot_\infty$ closure of any Schatten ideal and the only closed non-trivial two-sided $*$-ideal in $\Li$.

In order to simplify the notation throughout the paper we will use the indices $\pm$ or $\pm\pm$ to denote sets of operators acting in $\H_\pm$. For example $L^1_+$ would stand for $L^1(\H_+)$ and $L^2_{+-}$ would stand for $L^2(\H_-,\H_+)$ (note the order of signs in the last symbol). One can also introduce a block decomposition with respect of the polarization of an operator $A$ acting on $\H$:
\be \label{blocks}
A=\left(
\begin{array}{cc}
  A_{++} & A_{+-}\\
  A_{-+} & A_{--}\\
\end{array}
\right).\ee
We will also identify the operators $A_{\pm\pm}:\H_\pm\to\H_\pm$ with $P_\pm A P_\pm$ when it won't lead to any confusion.

\begin{defn}\label{def:gr}
The \emphh{restricted Grassmannian} $\Gr$ is defined as a set of closed subspaces $W\subset\H$ such that:
\ben[i)]
\item the orthogonal projection $p_+:W\to \H_+$ is a Fredholm operator;
\item the orthogonal projection $p_-:W\to \H_-$ is a Hilbert--Schmidt operator.
\een
\end{defn}

As said before, we identify a closed subspace $W$ with an orthogonal projector onto this subspace, which we denote $P_W$. 
Thus we identify $\Gr$ with the set of projectors $\{ P_W \;|\; W\in\Gr\}\subset \L$.

From \cite{spera-valli} one has the following equivalent description of $\Gr$:
\begin{prop}\label{prop:gr-proj}
\[W\in\Gr \Longleftrightarrow P_W-P_+ \in L^2\]
\end{prop}

While studying $\Gr$ certain Banach Lie groups and Banach Lie algebras appear. Their elements satisfy a condition $[X, P_+]\in L^2(\H)$. It means that off-diagonal blocks $X_{-+}$, $X_{+-}$ of block decomposition \eqref{blocks} are $L^2$ class. Note that due to the fact that $L^2(\H)$ is an ideal, composition of operators satisfying this condition also satisfies this condition.

One introduces the unitary restricted group:
\be\Ur:=\{ u\in \U \;|\; [u,P_+]\in L^2\},\ee
where $\U$ is a unitary group of operators acting in $\H$. It possesses the structure of Banach Lie group with Banach Lie algebra
\be\ur:=\{ x\in \uu \;|\; [x,P_+]\in L^2\},\ee
where $\uu$ is a Banach Lie algebra of skew symmetric operators on $\H$. $\Ur$  possesses also the structure of Banach Poisson--Lie group, see \cite{tumpach-bruhat}.

The Banach Lie group $\Ur$ acts transitively on $\Gr$ and the stabilizer of $\H_+$ for this action is $U_+\times U_-$, which is a Banach Lie subgroup of $\Ur$, see \cite[Definition 4.1]{beltita2006}. In this way, $\Gr$ can be seen as a smooth homogeneous space $\Ur/(U_+\times U_-)$.

The differential structure on $\Gr$ is obtained using the charts given formally by the same formulas as in the case of $\L$. One notes though that in case of $\Gr$ these charts take value in the Hilbert spaces $L^2(W,W^\perp)$ and transition functions are smooth with respect to the $L^2$ topology, see also \cite{segal,wurzbacher}. For clarity we will now write down some of the formulas in more details.

Analogously to the previous situation, the inverse of the chart $\phi_{W}^{-1}:L^2(W,W^\perp)\to \Gr$ assigns to a Hilbert--Schmidt operator the projector onto its graph in $W\oplus W^\perp = \H$. For the particular case of $W=\H_+$ this map assumes the form
\be \label{coord}
\phi_{\H_+}^{-1}(A) = \left(
\begin{array}{cc}
  (1+A^*A)^{-1} & (1+A^*A)^{-1}A^*\\
  A(1+A^*A)^{-1} & A(1+A^*A)^{-1}A^*\\
\end{array}
\right)
\ee
for $A\in L^2_{+-}$. This formula follows by expanding the projector $\phi_{\H_+}^{-1}(A)$ into block form and solving the following equations
\[ \phi_{\H_+}^{-1}(A) 
\left(\begin{array}{c} x\\ Ax\end{array}\right)
=
\left(\begin{array}{c} x\\ Ax\end{array}\right),
\]
\[ \phi_{\H_+}^{-1}(A) 
\left(\begin{array}{c} -A^*x\\ x\end{array}\right)
= 0
\]
for any $x\in \H_+$.
One easily sees using the Proposition \ref{prop:gr-proj} that indeed $\phi_{H_+}^{-1}$ takes values in $\Gr$. Namely, upper right corner of the block decomposition of 
$\phi_{\H_+}^{-1}(A)-P_+$ reads $(1+A^*A)^{-1}-1=-(1+A^*A)^{-1}A^*A$. It is easily seen that this and all other terms belong to $L^2$ class.

One can rewrite the formula \eqref{coord} as:
\be \phi_{\H_+}^{-1}(A) = (1_{\H_+} + A)(1_{\H_+}+A^*A)^{-1}(P_+ + A^* P_-).\ee
As a consequence from homogeneity we get that for an arbitrary $W\in\Gr$ the inverse of the chart can be written as
\be \phi_W^{-1}(A) = (1_W + A)(1_W+A^*A)^{-1}(P_W + A^* P_{W^\perp})\ee
for $A\in L^2(W,W^\perp)$.

Now, the formula \eqref{chart_grass} adapted to the case of the chart $\phi_{\H_+}:\Omega_{\H_+}\to L^2_{+-}$ can be expressed as
\be \phi_{\H_+}(p) = P_- p P_+ (P_+ p P_+)^{-1} = p_{-+}(p_{++})^{-1} = p(p_{++})^{-1}-P_+,\ee
where
\be \Omega_{\H_+} = \{ p\in \Gr \;|\; p_{++} \textrm{ is invertible in }\H_+\}.\ee
This map indeed takes values in $L^2_{+-}$ since $P_-p$ is Hilbert--Schmidt operator by the definition of the restricted Grassmannian.
% \todo{przydałby się komentarz, dlaczego wartości są w $L^2$}

Thus by direct computation we get that the transition maps for the atlas $\{\phi_{W},\Omega_W\}_{W\in \Gr}$ are equal
\be \psi_{\H_+,W}(A) = \phi_{\H_+}\circ\phi_W^{-1}(A) = \ee
\[P_-(1_W + A)(1_W+A^*A)^{-1}(P_W + A^* P_{W^\perp})P_+\big(P_+(1_W + A)(1_W+A^*A)^{-1}(P_W + A^* P_{W^\perp})P_+\big)^{-1}\]
for $A\in L^2(W,W^\perp)$ such that $\phi_W^{-1}(A)\in\Omega_{\H_+}$. This condition on the operator $A$ means that both the projectors $P_W$ and $P_+$ restricted to the graph of $A$ are invertible. Explicitly it is equivalent to the fact that the operator $P_+(1_W + A)(1_W+A^*A)^{-1}(P_W + A^* P_{W^\perp})_{\H_+}$ is invertible in $\H_+$. Observe that this operator is a product of three terms and the middle one $(1_W+A^*A)^{-1}$ is always invertible in $W$. Moreover operator $(1_W + A)$ is an isomorphism from $W$ to the graph of $A$. In consequence the operator $P_+(1_W + A):W\to P_+$, under the stated assumptions on $A$, is also invertible. It follows now that third term is also invertible and inverting the expression term-by-term we obtain
%$(P_W + A^* P_{W^\perp})P_+$ is invertible as a map from $\H_+$ to its image in $W$. However the image of $(P_W + A^* P_{W^\perp})P_+$ 
\be \label{trans} \psi_{\H_+,W}(A) = P_-(1_W + A)\big(P_+(P_W + A)\big)^{-1}.\ee
As expected, it is a fractional map and in consequence it is smooth as a composition of the inverse map in the space $L^\infty(W,\H_+)$, which is smooth, and multiplication $L^2(W,\H_-) \times L^\infty(\H_+,W)\to L^2(\H_+,\H_-)$, which is linear and bounded.

%Let us stress at this point that $^{-1}$ in formula \eqref{trans} denotes inverse in the space $L^\infty(W,P_+)$ and in consequence it is smooth as
Note that  partial inverse (or Moore--Penrose inverse) %considered in \cite{OS}, which 
is in general not smooth and even not continuous, see the discussion in \cite{BGJP} and references therein.

% \todo{Czy widać łatwo, że $\psi_{\H_+,W}(A)$ należy do $L^2_{++}$? Chyba ciągłość i gładkość wymagają więcej komentarza, bo to odwzorowanie $L^2\to L^\infty \to L^2$, więc mogłoby być nawet nieciągłe. Segal-Wilson i Wurzbacher zbywają ten temat, piszą nadto, że to odwzorowanie jest holomorficzne.}
\section{Banach Lie groupoid of partial isometries over restricted Grassmannian}\label{sec:res-group}

Since $\Gr$ is a subset of $\L$, one can define a groupoid 
$\UUr \gr \Gr$ as a subgroupoid of $\UU$ by
\be \label{gr-iso}\UUr = s^{-1}(\Gr) \cap t^{-1}(\Gr)\ee
or equivalently
\be \UUr = \{ u \in \UU \;|\; u^*u, uu^* \in \Gr\}. \ee
Directly from this definition it follows that this set is a subgroupoid of $\UU$ in the algebraic sense. However one still needs to introduce a differential structure on it to make it into a Banach Lie groupoid in the sense of Definition \ref{BLG}.

Before that, we note that the following proposition is a consequence of definition of $\UUr$:
\begin{prop}\label{prop:groupoid-L2}
For $u\in\UUr$ we have $[u,P_+]\in L^2$.
\end{prop}
\prf{
First of all the conclusion holds trivially for the case $u^*u=uu^*=P_+$, i.e. $u\in U_+$. Now since $\Ur$ acts transitively on $\Gr$, any $u\in\UUr$ can be written as $u=g_1 \tilde u g_2$ for $\tilde u\in U_+$ and $g_1, g_2\in \Ur$. 
The proposition follows since all operators on the right hands side of this expression have off-diagonal blocks of Hilbert--Schmidt class.
}

Let us note that unitary restricted group $\Ur$ is not contained in $\UUr$ (for example identity operator belongs to $\Ur$, but not to $\UUr$). That also shows that the conditions from the Proposition \ref{prop:groupoid-L2} above are not sufficient for $u$ to belong to $\UUr$.

In order to define a differential structure on $\UUr$ we need first the following fact:
\begin{prop}
For every point $W\in \Gr$ there exists a neighbourhood $\Omega_W\subset \Gr$ and a smooth map $\sigma_W: \Omega_W\to \Ur$ such that
\be \label{cross}\forall\; W'\in \Omega_W \;\;\; W' = \sigma_W(W') \H_+.\ee
\end{prop}
\prf{
We treat $\Gr$ as a homogeneous space $\Ur/(U_+\times U_-)$. If by $\pi:\Ur\to\Gr$ we denote the projection onto the orbits, then $\sigma_W$ is a local cross section of $\pi$, i.e. $\pi\circ\sigma = \id$.

Since $U_+\times U_-$ is a Banach Lie subgroup of $\Ur$, Theorem 4.19 from \cite{beltita2006} implies the result.}

Using a family of local cross sections we can transport every partial isometry to a unitary operator $\sigma_W(uu^*)^{-1} \;u\; \sigma_{W'}(u^*u)_{|\H_+}$ acting in the Hilbert space $\H_+$. In this way we 
construct a family of injective maps:
\be \UUr \supset s^{-1}(\Omega_{W'})\cap t^{-1}(\Omega_W) \tto \Gr \times \Gr \times U_+,\ee
for $W,W'\in\Gr$, given by
\be \label{urr-map} u \mapsto  (uu^*, u^*u, \sigma_W(uu^*)^{-1} \;u\; \sigma_{W'}(u^*u)_{|\H_+}),\ee
which assign to a partial isometry its final space, initial space and a unitary operator in $\H_+$. Naturally, knowing value of this map it is possible to reconstruct the partial isometry.

Now we couple these maps with charts on the respective manifolds. Namely let $(V_\alpha, \psi_\alpha)$ denote the atlas on $U_+$, where $V_\alpha\subset U_+$ are open sets and  $\psi_\alpha : V_\alpha \to \uu_+$. By $(\Omega_\beta, \sigma_\beta)$ denote a family of cross sections \eqref{cross} covering $\Gr$. Assume additionally that sets $\Omega_\beta$ are  chosen open and small enough that they are contained in chart domains of $\Gr$, i.e. there exists a map $\tilde\psi_\beta: \Omega_\beta \to L^2(W_\beta,W_\beta^\perp)$ for appropriately chosen $W_\beta\in\Omega_\beta\subset  \Gr$.

Denote by $\Omega_{\alpha\beta\gamma}$ the set
\be \Omega_{\alpha\beta\gamma} = \{ u \in \UUr \;|\; u^*u\in\Omega_\beta, uu^*\in\Omega_\gamma, (\sigma_\gamma(uu^*)^{-1} u \sigma_\beta(u^*u))_{|\H_+} \in V_\alpha\}\ee
and by $\Phi_{\alpha\beta\gamma}:\Omega_{\alpha\beta\gamma}\to L^2(W_\gamma,W_\gamma^\perp) \times L^2(W_\beta,W_\beta^\perp) \times \uu_+$ the map
%\be \Phi_{\alpha\beta\gamma} = (\tilde\psi_\gamma, \tilde\psi_\beta, \psi_\alpha) \circ \ee
\be \label{urr-chart}
\Phi_{\alpha\beta\gamma}(u) =  (\tilde\psi_\gamma(uu^*), \tilde\psi_\beta(u^*u), \psi_\alpha(\sigma_\gamma(uu^*)^{-1} u \sigma_\beta(u^*u))_{|\H_+}),\ee
which is composition of $(\tilde\psi_\gamma, \tilde\psi_\beta, \psi_\alpha)$ with the map \eqref{urr-map}.
%\[\in L^2(W_\gamma,W_\gamma^\perp) \times L^2(W_\beta,W_\beta^\perp) \times \uu_+ \]
\begin{thm}\label{thm:manifold}
The family $(\Omega_{\alpha\beta\gamma},\Phi_{\alpha\beta\gamma})$ constitutes a smooth atlas on the set of partial isometries $\UUr$ making it a smooth Banach manifold modelled on the Banach spaces $L^2(W_\gamma,W_\gamma^\perp) \times L^2(W_\beta,W_\beta^\perp) \times \uu_+$.
\end{thm}
\prf{
The collection of all sets $\Omega_{\alpha\beta\gamma}$ for all values of indices covers $\UUr$ and the maps $\Phi_{\alpha\beta\gamma}$ are injective. 
Their inverses can be written as
\be \label{urr-chart-inverse}\Phi_{\alpha\beta\gamma}^{-1}(A,B,X) = \sigma_\gamma(\tilde\psi_\gamma^{-1}(A))\psi_\alpha^{-1}(X)\sigma_\beta(\tilde\psi_\beta^{-1}(B))^{-1}\ee
for $A\in L^2(W_\gamma,W_\gamma^\perp)$, $B \in L^2(W_\beta,W_\beta^\perp)$, $X \in \uu_+ $ belonging to the image of $\Phi_{\alpha\beta\gamma}$. Note that we consider here $\psi_\alpha^{-1}(X)$ as an operator in $\H$ by extending it trivially from $\H_+$.

It remains to demonstrate that image is open and that the transition functions $\Phi_{\alpha\beta\gamma}\circ \Phi_{\alpha'\beta'\gamma'}^{-1}$ are smooth.

To prove openness of $\Phi_{\alpha\beta\gamma}(\Omega_{\alpha\beta\gamma})$ notice that the sets $\tilde\psi_\beta(\Omega_\beta)$, $\tilde\psi_\gamma(\Omega_\gamma)$, $\psi_\alpha(V_\alpha)$ are open. 
%So, given a triple $(A,B,X)$ in image of $\Phi_{\alpha\beta\gamma}$ it is possible to chose a neighbourhood which would be contained in $\tilde\psi_\beta(\Omega_\beta)\times\tilde\psi_\gamma(\Omega_\gamma)\times\psi_\alpha(V_\alpha)$. For any point in this neighbourhood formula \eqref{urr-chart-inverse} is well-defined and produces a partial isometry which by construction belongs to $\Omega_{\alpha\beta\gamma}$
For any point in $\tilde\psi_\beta(\Omega_\beta)\times\tilde\psi_\gamma(\Omega_\gamma)\times\psi_\alpha(V_\alpha)$ the right hand side of formula \eqref{urr-chart-inverse} is well-defined and produces a partial isometry which by construction belongs to $\Omega_{\alpha\beta\gamma}$. Thus we conclude that the image $\Phi_{\alpha\beta\gamma}(\Omega_{\alpha\beta\gamma})$ coincides with the Cartesian product
$\tilde\psi_\beta(\Omega_\beta)\times\tilde\psi_\gamma(\Omega_\gamma)\times\psi_\alpha(V_\alpha)$.

In order to prove smoothness of transition functions $\Phi_{\alpha\beta\gamma}\circ \Phi_{\alpha'\beta'\gamma'}^{-1}$ we will split them into three components:
\be\Phi_{\alpha\beta\gamma}\circ \Phi_{\alpha'\beta'\gamma'}^{-1}: (A,B,X) \mapsto (A',B',X'),\ee
where
\be A' = \tilde\psi_{\gamma}\circ\tilde\psi_{\gamma'}^{-1}(A),\ee
\be B' = \tilde\psi_{\beta}\circ\tilde\psi_{\beta'}^{-1}(B),\ee
\be X' = \psi_\alpha\big(\sigma_\gamma(\tilde\psi^{-1}_{\gamma'}(A))^{-1}  
\sigma_{\gamma'}(\tilde\psi_{\gamma'}^{-1}(A))\psi_{\alpha'}^{-1}(X)\sigma_{\beta'}(\tilde\psi_{\beta'}^{-1}(B))^{-1}
\sigma_\beta(\tilde\psi_{\beta'}(B))|_{\H+}\big).\ee
%where $\zeta$ is an inclusion map $L^\infty(\H_+)\subset \Li$.

The expressions for $A'$ and $B'$ are smooth since they are just transition maps for the manifold $\Gr$. The expression for $X'$ is a composition of smooth maps (between manifolds), operator multiplication and inversion map in the Banach Lie group $\Ur$. In consequence it is also smooth.
% \todo{sprawdzić kompletność dowodu}
}
Now we proceed to prove that $\UUr$ is a Banach Lie groupoid in the sense of definition \ref{BLG}. Note that source $s$ and target $t$ maps given by \eqref{u_source} and \eqref{u_target} in the introduced charts \eqref{urr-chart} are exactly projection onto first and second component of the Cartesian product. As such they are submersions.
Inversion map $\iota$ defined by \eqref{u_inv} in the chart assumes the form
\be \label{urr-chart-inversion}\Phi_{\beta\alpha'\gamma}\circ\iota\circ\Phi_{\alpha\beta\gamma}^{-1}(A,B,X) = (B,A,\psi_\alpha'((\psi_\alpha^{-1}(X))^*))\ee
and taking into account the fact that $*$ is an inversion map in the Banach Lie group $\Ur$, it is smooth. Now, multiplication map has the following formula in chart:
\be \label{urr-chart-mult}
\Phi_{\alpha''\beta'\gamma}(m(\Phi_{\alpha\beta\gamma}^{-1}(A,B,X), \Phi_{\beta\beta'\gamma'}^{-1}(B,C,X'))) =
(A,C,\psi_\alpha''(\psi_{\alpha}^{-1}(X)\psi_{\alpha'}^{-1}(X') )).\ee
Again since $\Ur$ is a Banach Lie group, the multiplication is smooth. 

Identity section $\epsilon$ \eqref{u_id} in the coordinates assumes the form
\be \label{urr-chart-identity}\Phi_{0\beta\beta}\circ\epsilon\circ\tilde\psi_{\beta}^{-1}(A) = (A,A,0),\ee
where $\psi_0$ is a particular chart on identity component of $U_+$ mapping unit element to $0\in\uu_+$.

Thus we have proved what follows:
%\sigma_\gamma(\tilde\psi_{\gamma'}(A))^{-1}  
%\iota(\sigma_{\gamma'}(\tilde\psi_{\gamma'}^{-1}(A))\psi_{\alpha'}^{-1}(X)\sigma_{\beta'}(\tilde\psi_{\beta'}^{-1}(B))) 

\begin{thm}\label{thm:groupoid}
The Banach manifold of partial isometries $\UUr$ is a Banach Lie groupoid with respect to the maps \eqref{u_source}-\eqref{u_inv}.
\end{thm}

\begin{prop} 
Unlike $\UU$, the groupoid $\UUr$ is transitive and pure, i.e. the map $(s,t):\UUr\to\Gr\times\Gr$ is surjective and both base and total space are modeled on a single (up to isomorphism) Banach space $L^2(\H_+,\H_-) \times L^2(\H_+,\H_-) \times \uu_+$.
\end{prop}

\begin{proof}
All subspaces $W\in\Gr$ are infinite dimensional and co-dimensional. Thus there exists a partial isometry mapping two elements of $\Gr$ to one another. Naturally this partial isometry is an element of $\UUr$.

Moreover for all $W\in\Gr$ the modeling space $L^2(W,W^\perp)$ is a separable Hilbert space, which is unique up to isomorphism.
\end{proof}

Note that $\UUr$ is not a Banach Lie subgroupoid of the Banach Lie groupoid of all partial isometries $\UU$. The obstacle is evident even on the level of bases of those groupoids. Namely, the restricted Grassmannian $\Gr$ is not a submanifold of the Grassmannian $\L$. It is only a weakly immersed submanifold (see \cite{BGJP}) as the image of the tangent of the inclusion map is not closed.

One can consider a family of more general $p$-restricted Grassmannians $\Grp$, for $1<p<\infty$ or $p=0$, by replacing $L^2$ space in the Definition \ref{def:gr} with $L^p$. Note that Grassmannian $\Gro$ was actually the first one to be introduced, see \cite{segal-wilson}. For all those Grassmannians it is possible to define a structure of Banach manifold in analogous way as in Section \ref{sec:grres}. 
\begin{prop}
    The sets $\UUr^p$ defined 
\be \label{gr-iso-p}\UUr^p = s^{-1}(\Grp) \cap t^{-1}(\Grp)\ee
are Banach Lie groupoids modelled on the Banach spaces $L^p(W,W^\perp) \times L^p(W',W'^\perp)\times \uu_+$.
\end{prop}
The proof is completely analogous to the proofs of Theorems \ref{thm:manifold} and \ref{thm:groupoid}. As before they are not Banach Lie subgroupoids due to the difference in topologies.

\begin{prop}
The Banach Lie groupoid $\UUr^0\gr\Gro$ is an immersed submanifold of $\UUr\gr\L$, but not split immersed submanifold.
\end{prop}
\begin{proof}
It follows from the fact that the ideal of compact operators $L^0$ is closed with respect to operator norm. However it does not admit a complementary Banach space, i.e. there is no such Banach space $E$ that $L^0\oplus E = L^\infty$.
\end{proof}

\backmatter

\bmhead{Acknowledgements}
This research was partially supported by National Science Centre, Poland/Fonds zur Förderung der wissenschaftlichen Forschung grant ``Banach Poisson--Lie groups and integrable systems'' number 2020/01/Y/ST1/00123. The authors have no relevant financial or non-financial interests to disclose.

%% BioMed_Central_Bib_Style_v1.01

\end{document}